\documentclass[a4paper,11pt]{article}
\input{xypic}
\usepackage{float,caption}
\usepackage{hyperref,bookmark,booktabs}
\usepackage{amsmath, amsthm, amssymb}
\usepackage[top=1.in, bottom=1in, left=1in, right=1in]{geometry}
\usepackage{graphicx} 
\usepackage{epstopdf} 
\usepackage{mdframed}
\usepackage{ragged2e}
\usepackage{mathrsfs}
\usepackage[shortlabels]{enumitem}
\usepackage{tikz-cd}
\usepackage{setspace}
\usepackage{stmaryrd}

\newtheorem{thm}{Theorem}[section]
\newtheorem{cor}[thm]{Corollary}
\newtheorem{lem}[thm]{Lemma}

\newtheorem{defn}[thm]{Definition}
\newtheorem{rem}[thm]{Remark}
\newtheorem{exam}[thm]{Example}

\theoremstyle{definition}

\def\P{\mathcal{P}}
\def\T{\mathcal{T}}
\def\h{\hat{\mathcal{T}}}
\def\s{\Sigma}
\def\bs{{\bf\s}}

\begin{document}
	\setcounter{page}{1}
	\title{Methods of generating soft topologies and soft separation axioms}
	\author{Zanyar A. Ameen$^{1*}$, Baravan A. Asaad$^{2,3}$, and Ramadhan A Mohammed$^4$
		\\
		$^1$Department of Mathematics, College of Science, University of Duhok,\\ Duhok 42001, Iraq; zanyar@uod.ac\\
		$^2$Deptartment of Mathematics, Faculty of Science, University of Zakho,\\ Duhok 42001, Iraq\\
		$^3$Department of Computer Science, College of Science, Cihan University-Duhok, Iraq\\
		 baravan.asaad@uoz.edu.krd\\
		$^4$Department of Mathematics, College of Basic Education, University of Duhok,\\ Duhok 42001, Iraq; ramadhan.hajani@uod.ac\\
}
	\date{}
	\maketitle
	
\begin{abstract}
The paper develops a novel analysis of mutual interactions between topology and soft topology. It is known that each soft topology produces a system of crisp (parameterized) topologies. The other way round is also possible. Namely, one can generate a soft topology from a system of crisp topologies. Different methods of producing soft topologies are discussed by implementing two formulas. Then, the relationships between the resulting soft topologies are obtained. With the help of an example, it is demonstrated that one formula is more constructible than the other. Now, it is reasonable to ask which (topological) properties of a soft topology can be transferred to the set of crisp topologies, or the opposite. To address this question, we consider the standard separation axioms and show how well these axioms can be preserved when moving from a system of crisp topologies to the soft topology generated by it and contrariwise. Additionally, our findings extend and disprove some results from the literature.
\end{abstract}
	
\section{Introduction}\label{sec1}
In its modern version, the Weierstrass Extreme Value Theorem demonstrates that topological considerations can be useful in decision-making theory and economics, (see, \cite{alcantud2022relationship}). Indeed, the development of topological structures helps to enhance other disciplines. 

General topology is the mathematical branch of topology that concerns itself with the foundational set-theoretic notions and constructions. Motivated by the standard axioms of classical topological space, Shabir and Naz \cite{shabir2011soft}, and {\c{C}}a{\u{g}}man et al. \cite{ccaugman2011soft}, separately, introduced another branch of topology named "soft topology."

Soft topology is a combination of soft set theory and topology. It is focused on the construction of the system of all soft sets.

Soft sets were presented as a collection of relevant parameters to characterize a universe of possibilities. Soft set theory has been a fruitful area of study and connection with various disciplines since its establishment. Molodtsov \cite{molodtsov1999soft}, in 1999, originated the soft set theory as a mathematical tool for dealing with uncertainty which is free of the challenges related with other theories such as fuzzy set theory \cite{ZADEH1965338}, rough set theory \cite{pawlak1982rough}, and so on. In particular, the nature of parameter sets associated with soft sets provides a standardized foundation for modeling uncertain data. This leads to the rapid growth of soft set theory and soft topology in a short amount of time and provides various applications of soft sets in real life. 

There are various studies that have made significant contributions to the development of soft topology since its foundation in \cite{ccaugman2011soft,shabir2011soft}. A soft topological approach was then used to interpret the behavior of the most fundamental concepts in (general) topology. To be specific, soft compactness \cite{aygunouglu2012some}, soft connectedness \cite{lin2013soft}, soft extremal disconnectedness \cite{asaad2017results}, soft submaximality \cite{alg2022soft}, soft simple extendedness \cite{ameen2022extensions}, and soft continuity \cite{zorlutuna2012remarks}.

Different methods of generating soft topologies on a common universal set were discussed in \cite{alcantud2020soft,alcantud2022relationship,nazmul2013neighbourhood,terepeta2019separating,zorlutuna2012remarks}.

Soft continuity of mappings has been widely generalized to diverse classes, including soft semi-continuity \cite{mahanta2012soft}, soft $\beta$-continuity \cite{yumak2015soft}, soft somewhat continuity \cite{ameen2021softsomewhat} and soft $\mathcal{U}$-continuity \cite{ameennon}.

Soft separation axioms are another significant aspect in the late development of soft topology; see for example \cite{asaad2022hypersoft,min2011note,shabir2011soft}.

Two remarkable formulas for generating soft topologies from a system of crisp topologies have been given by Terepeta \cite{terepeta2019separating}. One of the formulas (Formula 2) is said to generate a single set soft topology, while the other one generates a more general soft topology (Formula 1). Terepeta mainly applied Formula 2 to study the inheritance of soft separation axioms after the system of crisp topologies. Recently, Alcantud \cite{alcantud2020soft} proposed a slight extension of Formula 1 (we also call it Formula 1). He then employed such a formula to investigate the behavior of separability and second countability axioms between a system of crisp topologies and the soft topology generated by it. Very recently, Alcantud \cite{alcantud2022relationship} established crucial relationships between soft and fuzzy soft topologies. The work of Terepeta and Alcantud inspired us to attempt this research. Following their direction, we first apply the formulas to the system of crisp topologies taken from a soft topology in order to determine the connections between the obtained soft topologies and the original one. In addition, we use Formula 1 to verify how well the separation axioms are transferred between a system of crisp topologies and the soft topology that it generates. The latter statement extends the work of Terepeta (see, Section 2.1 in \cite{terepeta2019separating}), which is the main objective of this research.

\section{Preliminaries}\label{se2}\
Let $X$ be an initial universe, $\mathcal{P}(X)$ be all subsets of $X$ and $E$ be a set of parameters. An ordered pair $(F,E)=\{(e,F(e)):e\in E\}$ is said to be a soft set over $X$, where $F:E\to\P(X)$ is a set value mapping. The family of all soft sets on $X$ is represented by $S_E(X)$. The soft set $(X,E)\backslash (F,E)$ (or simply $(F,E)^c=(F^c,E)$) is the complement of $(F,E)$, where $F^c:E\to\mathcal{P}(X)$ is given by $F^c(e) = X\backslash F(e)$ for each $e\in E$. A soft subset $(F,E)$ over $X$ is called null, denoted by $\widetilde{\Phi}$, if $F(e)=\emptyset$ for each $e\in E$ and is called absolute, denoted by $\widetilde{X}$, if $F(e)=X$ for each $e\in E$. Notice that ${\widetilde{X}}^c=\widetilde{\Phi}$ and $\widetilde{\Phi}^c=\widetilde{X}$. It is said that $(A,E_1)$ is a soft subset of $(B,E_2)$ (written by $(A,E_1)\widetilde{\subseteq} (B,E_2)$, \cite{maji2003soft}) if $E_1\subseteq E_2$ and $A(e)\subseteq B(e)$ for each $e\in E_1\subseteq E$, and $(A,E_1)=(B,E_2)$ if $(A,E_1)\widetilde{\subseteq} (B,E_2)$ and $(B,E_2)\widetilde{\subseteq} (A,E_1)$. The union of soft sets $(A,E), (B,E)$ is represented by $(F,E)=(A,E)\widetilde{\cup} (B,E)$, where $F(e)=A(e)\cup B(e)$ for each $e\in E$, and intersection of soft sets $(A,E), (B,E)$ is given by $(F,E)=(A,E)\widetilde{\cap} (B,E)$, where $F(e)=A(e)\cap B(e)$ for each $e\in E$, see \cite{ali2009some}. A soft point \cite{shabir2011soft} is a soft set $(F,E)$ over $X$ in which $F(e) = \{x\}$ for each $e\in E$, where $x\in X$, and is denoted by $(\{x\},E)$. It is said that a soft point $(\{x\},E)$ is in $(F,E)$ (briefly, $x\in (F,E)$) if $x\in F(e)$ for each $e\in E$. On the other hand, $x\notin (F,E)$ if $x\notin F(e)$ for some $e\in E$. This implies that if $(\{x\},E)\widetilde{\bigcap}(F,E)=\widetilde{\Phi}$, then $x\notin (F,E)$.

	\begin{defn}\cite{shabir2011soft}\label{defntop}
		A collection $\s$ of $S_E(X)$ is said to be a soft topology on $X$ if the following conditions are satisfied:
		\begin{enumerate}[(i)]
			\item $\widetilde{\Phi},\widetilde{X}\in\s$;
			\item If $(F_1,E), (F_2,E)\in\s$, then $(F_1,E)\widetilde{\cap} (F_2,E)\in\s$; and
			\item If each $\{(F_i,E):i\in I\}\widetilde{\subseteq}\s$, then $\widetilde{\bigcup}_{i\in I} (F_i,E)\in\s$.
		\end{enumerate}
Terminologically, we call $(X,\s, E)$ a soft topological space on $X$. The elements of $\s$ are called soft open sets in $\s$ (or simply, soft open sets when no confusion arises), and their complements are called soft closed sets in $\s$ (or shortly, soft closed sets).
	\end{defn}

In what follows, by $(X,\s, E)$ we mean a soft topological space, and by disjoint of two soft sets $(F,E), (G,E)$ over $X$ we mean $(F,E)\widetilde{\cap}(G,E)=\widetilde{\Phi}$.

\begin{defn}\cite{ccaugman2011soft}
A subcollection $\mathcal B\subseteq\s$ is called a soft base for the soft topology $\s$ if each element of $\s$ is a union of elements of $\mathcal B$.
\end{defn}

\begin{defn}\cite{ccaugman2011soft}
	Let $\s_1, \s_2$ be two soft topologies on $X$. It is said that $\s_2$ is finer than $\s_1$ (or $\s_1$ is coarser than $\s_2$) if $\s_1\widetilde{\subseteq}\s_2$ 
\end{defn}

\begin{lem}\label{crisp}\cite{shabir2011soft}
Let $(X,\s, E)$ be a soft topology on $X$. For each $e\in E$, $\s_{e}=\{F(e):(F,E)\in\s\}$ is a crisp topology on $X$.
\end{lem}
	
\begin{defn}\cite{ameen2022minimal}
Let $\mathcal{F}\widetilde{\subseteq}S_E(X)$. The intersection of all soft topologies on $X$ including $\mathcal{F}$ is called a soft topology generated by $\mathcal{F}$ and is referred to $T(\mathcal{F})$.  
\end{defn}

\begin{lem}\cite[Lemma 3.5]{alghour2022maximal}\label{zlemma}
	Let $\s_1, \s_2$ be two soft topologies on $X$. The resulting soft topology $T(\s_1\widetilde{\cup} \s_2)$ is identical to the soft topology $T(\mathcal{F})$ generated by $\mathcal{F}=\{(F_1,E)\widetilde{\cap}(F_2,E):(F_1,E)\in\s_1,(F_2,E) \in\s_2\}$.
\end{lem}

\begin{defn}\label{T_i}\cite{shabir2011soft}
	A soft space $(X, \s, E)$ is called
	\begin{enumerate}[(i)]
		\item soft $T_0$ if for each $x,y\in X$ with $x\ne y$, there exist soft open sets $(U,E), (V,E)$ such that $x\in (U,E)$, $y\notin  (U,E)$ or $x\notin  (V,E)$, $y\in  (V,E)$,
		\item soft $T_1$ if for each $x,y\in X$ with $x\ne y$, there exist soft open sets $(U,E), (V,E)$ such that $x\in (U,E)$, $y\notin  (U,E)$ and $x\notin  (V,E)$, $y\in  (V,E)$,
		\item soft $T_2$ $($soft Hausdorff$)$ if for each $x,y\in X$ with $x\ne y$, there exist soft open sets $(U,E), (V,E)$ containing $x, y$ respectively such that $(U,E)\widetilde{\bigcap} (V,E)=\widetilde{\Phi}$. 
		\item soft regular if for each soft closed set $(F,E)$ and each soft point $x$ with $x\notin (F,E)$, there exist soft open sets $(U,E), (V,E)$ such that $x\in  (U,E)$, $(F,E)\widetilde{\subseteq} (V,E)$ and $(U,E)\widetilde{\bigcap} (V,E)=\widetilde{\Phi}$.
		\item soft normal if for each soft closed sets $(F,E),(D,E)$ with $(F,E)\widetilde{\bigcap}(D,E)=\widetilde{\Phi}$, there exist soft open sets $(U,E), (V,E)$ such that $(F,E)\widetilde{\subseteq} (U,E)$, $(D,E)\widetilde{\subseteq} (V,E)$ and $(U,E)\widetilde{\bigcap} (V,E)=\widetilde{\Phi}$.
		\item soft $T_3$ if it is soft $T_1$ and soft regular.
		\item soft $T_4$ if it is soft $T_1$ and soft normal.
	\end{enumerate}
\end{defn}

\begin{lem}\cite[Theorem 3.18]{min2011note}\label{te1=te2}
	If $(X,\s,E)$ is a soft regular space, then $\s_e=\s_{e'}$ for each $e,e'\in E$.
\end{lem}

\section{Methods of generating soft topologies and their relationships}
This section provides different methods of producing soft topologies via Formulas 1 \& 2. An example is given which discusses the implementation of these formulas in detail. The relationships between the original soft topology and the soft topologies that are produced by Formulas 1 \& 2.
\begin{defn}\cite{alcantud2020soft,terepeta2019separating}\label{defn1}
Let $\bs=\{\s_e\}$ be a family of (crisp) topologies on a set $X$ indexed by $E$. Then following procedures produce different soft topologies on $X$:
\begin{equation*}
(\text{Formula } 1)\hspace{1cm} \T(\bs)=\Big\{\{(e,F(e)):e\in E\}\in S_E(X):F(e)\in \s_e, \forall e\in E \Big\},\\
\end{equation*}
$\T(\bs)$ is called a soft topology generated by $\bs$. If for each $e,e'\in E$, $\s_e=\s_{e'}=\s$, then $\T(\bs)=\T(\s)$.

\begin{equation*}
	(\text{Formula } 2)\hspace{.5cm} \h(\s_e)=\Big\{\{(e,F(e)):e\in E\}\in S_E(X):F(e)=F(e')\in\s_e, \forall e,e'\in E \Big\},\\
\end{equation*}
$\h(\s_e)$ is called a single set soft topology generated by $\s_e$. 
\end{defn}

\begin{defn}\label{asso}
Let $(X,\s,E)$ be a soft topological space. If $\bs=\{\s_e:e\in E\}$ is the family of all crisp topologies from $\s$, then we call $\T(\bs)$ the soft topology associated with $\s$.

Note that $\T(\bs)$ is called an extended soft topology in \cite{nazmul2013neighbourhood}.
\end{defn}

\begin{lem}\label{l1}
Let $\bs=\{\s_e:e\in E\}$ be the family of all crisp topologies from $(X,\s,E)$. Then $$\s\widetilde{\subseteq}\T(\bs).$$
\end{lem}
\begin{proof}
It can be concluded from the definition of soft sets and the soft topology generated by $\bs$.
\end{proof}

\begin{lem}\label{base(e)}
Let $\bar{\beta}=\{\beta_e:e\in E\}$ be a family of bases for the topologies $\s_e$ on $X$. Then $\mathcal{B}(\bar{\beta})=\Big\{\{(e,F(e)):e\in E\}\in S_E(X):F(e)\in\beta_e\cup\{\emptyset\}, \forall e\in E \Big\}$ is a base for a soft topology on $X$ and $\T(\bs)=T(\mathcal{B}(\bar{\beta}))$.
\end{lem}
\begin{proof}
By using Corollary 3 in \cite{alcantud2020soft} and simple modifications to the proof of Theorem 3 in \cite{alcantud2020soft}, we can conclude the proof.
\end{proof}

The following is a straightforward generalizations of Lemma \ref{zlemma}, thus the proof is omitted:
\begin{lem}\label{zlemma2}
	Let $\{\s_e:e\in E\}$ be a family of soft topologies on $X$. The resulting soft topology $T(\widetilde{\bigcup}_{e\in E} \s_e)$ is identical to the soft topology $T(\mathcal{F})$ generated by $\mathcal{F}=\{\widetilde\bigcap_{{e_i}=1}^n(F_{e_i},E):(F_{e_i},E)\in \widetilde{\bigcup}_{{e_i}\in E}\s_{e_i}\}$.
\end{lem}

\begin{lem}\label{l2}
Let $\bs=\{\s_e:e\in E\}$ be a family of crisp topologies on $X$. Then 
\begin{equation*}
	T(\widetilde{\bigcup}_{e\in E}\h(\s_e))=\h(T(\bigcup_{e\in E}\s_e)).
\end{equation*}
\end{lem}
\begin{proof}
The Lemma \ref{base(e)} reduces the task of working with basic soft open sets rather than soft open sets. Let $(B_0,E)\in T(\widetilde{\bigcup}_{e\in E}\h(\s_e))$. Then $(B_0,E)=\widetilde{\bigcap}_{i=1}^n(B_i,E)$ for $(B_i,E)\in\widetilde{\bigcup}\h(\s_e)$, and so $(B_0,E)=\widetilde{\bigcap}_{i=1}^n(B_i,E)$ such that $(B_i,E)\in\h(\s_e)$ for some $e\in E$. By Formula 2, one can detach $E$ from $(B_i,E)$ for $i=0,1,\cdots, n$, and get $B_0=\bigcap_{i=1}^nB_i$, where  $B_i\in\s_e$ for some $e\in E$. This implies that $B_0=\bigcap_{i=1}^nB_i$ for  $B_i\in\bigcup\s_e$. Therefore, by Formula 2, $(B_0,E)\in\h(T(\widetilde{\bigcup}\s_e))$. The reverse of the inclusion can be proved by a similar technique.
\end{proof}

The following example shows how the techniques in Definition 1 and the relations in Lemmas \ref{l1}-\ref{l2}  can be used in practice:
\begin{exam}\label{exam1}
Let $X=\{x_1,x_2,x_3\}$, $E=\{e_1,e_2\}$. Consider the soft topology on $X$, $$\s=\{\widetilde{\Phi}, (F_1,E), (F_2,E), (F_3,E), (F_4,E), \widetilde{X}\},$$ where
\begin{flalign*} 
(F_1,E)&=\{(e_1,\{x_1\}),(e_2,\emptyset)\},&&\\
(F_2,E)&=\{(e_1,\{x_1,x_2\}),(e_2,X)\},&&\\
(F_3,E)&=\{(e_1,\emptyset),(e_2,\{x_3\})\},\text{ and}&&\\
(F_4,E)&=\{(e_1,\{x_1\}),(e_2,\{x_3\})\}.&&
\end{flalign*}

The crisp topologies from $\s$ are 
\begin{center}
$\s_{e_1}=\{\emptyset, \{x_1\}, \{x_1,x_2\}, X\}$ and $\s_{e_2}=\{\emptyset, \{x_3\}, X\}$.
\end{center}

Applying the Formula 2, we obtain the following two soft topologies on $X$:
\begin{flalign*}
\h(\s_{e_1})&=\Big\{\widetilde{\Phi}, \{(e_1,\{x_1\}), (e_2,\{x_1\})\}, \{(e_1,\{x_1,x_2\}), (e_2,\{x_1,x_2\})\}, \widetilde{X}\Big\}&&\\ 
&=\Big\{\widetilde{\Phi}, (\{x_1\},E), (\{x_1,x_2\},E), \widetilde{X}\Big\}\text{ (more compactly) and}\\
\h(\s_{e_2})&=\Big\{\widetilde{\Phi}, \{(e_1,\{x_3\}), (e_2,\{x_3\})\}, \widetilde{X}\Big\}=\{\widetilde{\Phi}, (E,\{x_3\}), \widetilde{X})\}.&&
\end{flalign*}

From Lemma \ref{zlemma2}, we can naturally generate a soft topology $T$ on $X$ by the union of $\h(\s_{e_1})$ and $\h(\s_{e_2})$. That is,
$$T\Big(\widetilde{\bigcup}_{i=1}^2\h(\s_{e_i})\Big)=\Big\{\widetilde{\Phi}, (G_1,E), (G_2,E), (G_3,E), (G_4,E), \widetilde{X}\Big\},$$ where
\begin{flalign*}
(G_1,E)&=\{(e_1,\{x_1\}), (e_2,\{x_1\})\},&&\\
(G_2,E)&=\{(e_1,\{x_3\}), (e_2,\{x_3\})\},&&\\
(G_3,E)&=\{(e_1,\{x_1,x_2\}), (e_2,\{x_1,x_2\})\},\text{ and}&&\\
(G_4,E)&=\{(e_1,\{x_1,x_3\}), (e_2,\{x_1,x_3\})\}.&&
\end{flalign*}
The compact form of the above conclusion is
 $$T\Big(\widetilde{\bigcup}_{i=1}^2\h(\s_{e_i})\Big)=\Big\{\widetilde{\Phi}, (\{x_1\},E), (\{x_3\},E), (\{x_1, x_2\},E), (\{x_1, x_3\},E), \widetilde{X}\Big\}.$$

By applying the Formula 1, the next soft topology on $X$ will be obtained.
\begin{center}
$\T(\bs)=\T(\{\s_{e_1},\s_{e_2}\})=\Big\{\widetilde{\Phi}, (H_1,E), (H_2,E), \cdots, (H_{10},E), \widetilde{X}\Big\}$,
\end{center}
 where
 \vspace{-3mm}
\begin{flalign*}
(H_1,E)&=\{(e_1,\emptyset), (e_2,X)\},&&\\
(H_2,E)&=\{(e_1,\emptyset), (e_2,\{x_3\})\},&&\\
(H_3,E)&=\{(e_1,X), (e_2,\emptyset)\},&&\\
(H_4,E)&=\{(e_1,X), (e_2,\{x_3\})\},&&\\
(H_5,E)&=\{(e_1,\{x_1\}), (e_2,\emptyset)\},&&\\
(H_6,E)&=\{(e_1,\{x_1\}), (e_2,X)\},&&\\
(H_7,E)&=\{(e_1,\{x_1\}), (e_2,\{x_3\})\},&&\\
(H_8,E)&=\{(e_1,\{x_1, x_2\}), (e_2,\emptyset)\},&&\\
(H_9,E)&=\{(e_1,\{x_1, x_2\}), (e_2,X)\},\text{ and}&&\\
(H_{10},E)&=\{(e_1,\{x_1,x_2\}), (e_2,\{x_3\})\}.&&
\end{flalign*}

One can easily check that the above computations lead to the following observations:
\begin{itemize}
	\item $\s$ is a subcollection of $\T(\bs)$.
	\item  $\s$ is independent of $T\big(\widetilde{\bigcup}_{i=1}^2\h(\s_{e_i})\big)$.
	\item  $T\big(\widetilde{\bigcup}_{i=1}^2\h(\s_{e_i})\big)$  is independent of $\T(\bs)$.
\end{itemize}
\end{exam}

The relationships between the soft topologies on a common universe obtained by the methods described in this section is summarized in Diagram 1. 
\begin{figure}[h!]
\captionsetup{name=Diagram}
\centering
\begin{tikzcd}[row sep=5em,column sep=6em]
	&\s \arrow[d,"\text{produces}",sloped] \arrow[dl,"\widetilde{\supseteq}",sloped] \arrow[dr,"\text{independent}",sloped,dash] &
	\\
	\T(\bs) & \arrow[l,"generates",swap] \bs=\{\s_e:e\in E\} \arrow[r,"generates"] & \h(T(\widetilde{\bigcup}\bs)) 
\end{tikzcd}	
\caption[Diagram 1]{Relationships between different soft topologies}
\end{figure}
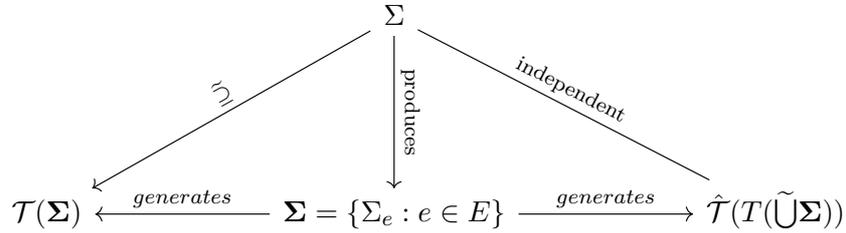

\section{Non-uniqueness of soft topology $\T(\bs)$ associated with $\s$}
In this short section, we provide an example to witness that two different (even incomparable) soft topologies may have a common associated soft topology.
\begin{exam}\label{exam3}
Consider the soft topology on $X$, $\s=\{\widetilde{\Phi}, (F_1,E), (F_2,E), (F_3,E), (F_4,E), \widetilde{X}\}$ given in Example \ref{exam1}, there $X=\{x_1,x_2,x_3\}$, $E=\{e_1,e_2\}$. Let $\s'=\{\widetilde{\Phi}, (R_1,E)$, $(R_2,E)$, $(R_3,E)$, $(R_4,E),\widetilde{X}\}$ be another soft topology on $X$, where 
\begin{flalign*}
(R_1,E)&=\{(e_1, \{x_1\}),(e_2,\emptyset)\},&&\\
(R_2,E)&=\{(e_1, \{x_1,x_2\}),(e_2,X)\},&&\\
(R_3,E)&=\{(e_1, X),(e_2,\{x_3\})\},\text{ and}&&\\
(R_4,E)&=\{(e_1, \{x_1,x_2\}),(e_2,\{x_3\})\}.&&
\end{flalign*}
Then $\s$ and $\s'$ are incomparable. Set $\hat{\s}=\s\widetilde{\bigcup}\s'$. Therefore,  $\hat{\s}$ is finer than both $\s$ and $\s'$. On the other hand, $\s, \s'$ and $\hat{\s}$ have the same family of crisp topologies $\bs=\{\s_{e_1},\s_{e_2}\}$, and thus they generate only one $\T(\bs)$.
\end{exam}

\section{Separation axioms preservation between $\bs$ and $\T(\bs)$}
With the exception of $\T(\s)$, a single set soft topology $\h(\s)$ generated by $\s$ inherits soft separation axioms after $\s$, according to Terepeta \cite{terepeta2019separating}. In this section, we use  $\T(\bs)$ to see how well separation axioms are preserved when moving from $\bs$ to $\T(\bs)$ and vice versa.

Here, we start defining special types of soft sets that exist when constructing a soft topology by using $\T(\bs)$ for future usage. 
\begin{defn}\label{zset}
Let $G,H\subset X$ and let $e\in E$, we define 
\begin{enumerate}[(i)]
\item $(F^G_e,E)$ to be a soft set over $X$ such that $F^G_e(e)=G$ and $F^G_e(e')=X$ for each $e'\ne e$.
\item $(F^e_H,E)$ to be a soft set over $X$ such that $F^e_H(e)=H$ and $F^e_H(e')=\emptyset$ for each $e'\ne e$, (see, \cite[Definition 5]{alcantud2020soft}).
\end{enumerate}
Note that $(F^G_e,E)^c=(F^e_H,E)$ if and only if $G^c=H$.
\end{defn}

\begin{thm}\label{thm1}
Let $\bs=\{\s_e:e\in E\}$ be a family of crisp topologies on $X$. If $\s_e$ is a $T_0$-space for some $e\in E$ then $\T(\bs)$ is a soft $T_0$-space.
\end{thm}
\begin{proof}
Suppose that $\s_e$ is a $T_0$-space for some $e\in E$. Let $x,y\in X$ with $x\neq y$. Then there exist open sets $U,V\in \s_e$ such that $x\in U$, $y\notin U$ or $x\notin V$, $y\in V$. By Definition \ref{zset}, there exist two corresponding soft sets $(F^U_e,E), (F^V_e,E)\in \T(\bs)$ for which $x\in (F^U_e,E)$, $y\notin (F^U_e,E)$ or $x\notin (F^V_e,E)$, $y\in (F^V_e,E)$. Thus, $\T(\bs)$ is soft $T_0$.
\end{proof}

The following example shows that the converse of Theorem \ref{thm1} is not true in general.
\begin{exam}\label{ex1}
Let $X=\{x_1,x_2,x_3\}$ and let $\s_{e_1}=\{\emptyset, \{x_1\}, \{x_2,x_3\}, X\}$ and $\s_{e_2}=\{\emptyset, \{x_3\},$ $\{x_1,x_2\}, X\}$ be crisp topologies on $X$ indexed by $E=\{e_1,e_2\}$. By using the Formula 1, the following soft topology on $X$ will be obtained:
\begin{center}
	$\T(\bs)=\T(\{\s_{e_1},\s_{e_2}\})=\Big\{\widetilde{\Phi}, (H_1,E), (H_2,E), (H_3,E), \cdots, (H_{15},E), \widetilde{X}\Big\}$,
\end{center}
where
\vspace{-3mm}
\begin{flalign*}
	(H_1,E)&=\{(e_1,X), (e_2,\emptyset)\},&&\\
	(H_2,E)&=\{(e_1,\{x_1\}), (e_2,\emptyset)\},&&\\
	(H_3,E)&=\{(e_1,\{x_2,x_3\}), (e_2,\emptyset)\},&&\\
	(H_4,E)&=\{(e_1,\emptyset), (e_2,X)\},&&\\
	(H_5,E)&=\{(e_1,\{x_1\}), (e_2,X)\},&&\\
	(H_6,E)&=\{(e_1,\{x_2,x_3\}), (e_2,X)\},&&\\
	(H_7,E)&=\{(e_1,\emptyset), (e_2,\{x_1,x_2\})\},&&\\
	(H_8,E)&=\{(e_1,X), (e_2,\{x_1,x_2\})\},&&\\
	(H_9,E)&=\{(e_1,\{x_1\}), (e_2,\{x_1,x_2\})\},&&\\
	(H_{10},E)&=\{(e_1,\{x_1\}), (e_2,\{x_2\})\},&&\\
	(H_{11},E)&=\{(e_1,\{x_2,x_3\}), (e_2,\{x_1,x_2\})\},&&\\
	(H_{12},E)&=\{(e_1,\emptyset), (e_2,\{x_3\})\},&&\\
	(H_{13},E)&=\{(e_1,X), (e_2,\{x_3\})\},&&\\
	(H_{14},E)&=\{(e_1,\{x_1\}), (e_2,\{x_3\})\},\text{ and}&&\\
	(H_{15},E)&=\{(e_1,\{x_2,x_3\}), (e_2,\{x_3\})\}.&&
\end{flalign*}
Easily one can check that $\T(\bs)$ is soft $T_0$. On the other hand, neither of $\s_{e_1}$ nor $\s_{e_2}$ is $T_0$.
\end{exam}

\begin{thm}\label{thm2}
Let $\bs=\{\s_e:e\in E\}$ be a family of crisp topologies on $X$. If $\s_e$ is a $T_1$-space for some $e\in E$ then $\T(\bs)$ is a soft $T_1$-space.
\end{thm}
\begin{proof}
It is entirely analogous to the first part of the proof of Theorem \ref{thm1}.
\end{proof}

The following example shows that the converse of Theorem \ref{thm2} is not true in general. It also refutes Theorem 3.5 in \cite{goccur2015soft}:
\begin{exam}\label{exam2}
Let $X=\{x_1,x_2\}$ and let $\s_{e_1}=\{\emptyset, \{x_1\}, X\}$ and $\s_{e_2}=\{\emptyset, \{x_2\}, X\}$ be crisp topologies on $X$ indexed by $E=\{e_1,e_2\}$. By using the Formula 1, the following soft topology on $X$ will be obtained:
	\begin{center}
		$\T(\bs)=\T(\{\s_{e_1},\s_{e_2}\})=\Big\{\widetilde{\Phi}, (H_1,E), (H_2,E), (H_3,E), (H_4,E), (H_5,E), (H_6,E), (H_7,E), \widetilde{X}\Big\}$,
	\end{center}
	where
	\vspace{-3mm}
	\begin{flalign*}
		(H_1,E)&=\{(e_1,\emptyset), (e_2,X)\},&&\\
		(H_2,E)&=\{(e_1,X), (e_2,\emptyset)\},&&\\
		(H_3,E)&=\{(e_1,\emptyset), (e_2,\{x_2\})\},&&\\
		(H_4,E)&=\{(e_1,X), (e_2,\{x_2\})\},&&\\
		(H_5,E)&=\{(e_1,\{x_1\}), (e_2,\emptyset)\},&&\\
		(H_6,E)&=\{(e_1,\{x_1\}), (e_2,X)\}, \text{ and}&&\\
		(H_7,E)&=\{(e_1,\{x_1\}), (e_2,\{x_2\})\}.&&
	\end{flalign*}
Then $(H_4,E)$ and $(H_6,E)$ are soft open sets in $\T(\bs)$ such that $x_1\in(H_6,E)$, $x_2\notin(H_6,E)$ and $x_2\in(H_4,E)$, $x_1\notin(H_4,E)$. Thus, $\T(\bs)$ is soft $T_1$. On the other hand, neither of $\s_{e_1}$ nor $\s_{e_2}$ is $T_1$.
\end{exam}

\begin{thm}\label{thm3}
Let $\bs=\{\s_e:e\in E\}$ be a family of crisp topologies on $X$. Then $\s_e$ is a $T_2$-space for each $e\in E$ if and only if $\T(\bs)$ is a soft $T_2$-space.
\end{thm}
\begin{proof}
Assume that $\s_e$ is $T_2$ for each $e\in E$. Let $x,y\in X$ with $x\neq y$. Then, for each $e$, there exist open sets $U(e),V(e)\in \s_e$ such that $x\in U(e)$, $y\in V(e)$ and $U(e)\cap V(e)=\emptyset$. Set $(U,E)=\{(e,U(e)):e\in E\}$ and $(V,E)=\{(e,V(e)):e\in E\}$. So $(U,E), (V,E)\in\T(\bs)$ such that  $x\in(U,E)$, $y\in(V,E)$ and $(U,E)\widetilde{\bigcap}(V,E)=\{(e, U(e)\cap V(e)): e\in E\}=\widetilde{\Phi}$. Hence, $\T(\bs)$ is soft $T_2$.
	
Conversely, let $x,y\in X$ with $x\neq y$. Suppose that $\T(\bs)$ is soft $T_2$, then exists a soft open set $(G,E), (H,E)$ such that $x\in (G,E)$, $y\in (H,E)$ and $(G,E)\widetilde{\bigcap}(H,E)=\widetilde{\Phi}$. This means that, for each $e\in E$,  $x\in G(e)$, $y\in H(e)$ and $G(e)\cap H(e)=\emptyset$. Thus, $\s_e$ is $T_2$ for each $e\in E$.
\end{proof}

\begin{cor}\label{corthm3}
If $(X,\s,E)$ is a soft $T_2$-space, then $\s_e$ is $T_2$ for each $e\in E$.
\end{cor}
\begin{proof}
It is an immediate consequence of Lemma \ref{l1} and Theorem \ref{thm3}.
\end{proof}

Notice that Theorem \ref{thm3} and Corollary \ref{corthm3} generalize (part of) Theorem 4 in \cite{terepeta2019separating} and Proposition 17 in \cite{shabir2011soft}, respectively.

\begin{lem}\cite[Theorem 3]{terepeta2019separating}\label{closed=softclosed}
	Let $\s$ be a (crisp) topology on a set $X$. A soft set $(F,E)$ is soft closed in $\T(\s)$ if and only if $(F,E)=\{(e,F(e)):F^c(e)\in \s\}$.
\end{lem}

\begin{thm}\label{thm4}
Let $\bs=\{\s_e:e\in E\}$ be a family of crisp topologies on $X$. If $\T(\bs)$ is a soft regular space, then $\s_e$ is a regular space for each $e\in E$.
\end{thm}
\begin{proof}
Let $e\in E$. Take $x\in X$ and $F(e)$ be a closed set in $(X,\s_e)$ such that $x\notin F(e)$. The Definition \ref{defn1} and Lemma \ref{te1=te2} tell us that soft regularity of $\T(\bs)$ guarantees the equality of $\T(\bs)=\T(\s)$. Set $(F,E)=\{(e,F(e)):F^c(e)\in \s\}$. By Lemma \ref{closed=softclosed}, $(F,E)$ is soft closed in $\T(\s)$ along with $x\notin(F,E)$. Since $\T(\s)$ is soft regular, then there exist soft open sets $(U,E), (V,E)$ in $\T(\s)$ such that $x\in (U,E), (F,E)\widetilde{\subseteq}(V,E)$ and $\widetilde{\Phi}=(U,E)\widetilde{\bigcap}(V,E)=\{(e,U(e)\cap V(e)): e\in E\}$. This implies that $x\in U(e)$, $F(e)\subseteq V(e)$ and $U(e)\cap V(e)=\emptyset$ for each $e\in E$. Since $U(e), V(e)\in\s_e$, then $\s_e$ is regular for each $e\in E$.
\end{proof}

\begin{cor}\label{corthm4}
If $(X,\s,E)$ is a soft regular space, then $\s_e$ is regular for each $e\in E$.
\end{cor}
\begin{proof}
It can be concluded from Lemma \ref{l1} and Theorem \ref{thm4}.
\end{proof}

\begin{rem}
We shall mention that it is observed in Remark 3.23 (2') \cite{min2011note} that if $(X,\s,E)$ is a soft $T_3$ space, then $\s_e$ is $T_3$ for each $e\in E$. This conclusion is more general than Corollary \ref{corthm4}, but it cannot be followed from any of our results due to Example \ref{exam2}.
\end{rem}

The examples given below disprove the reverse of theorem \ref{thm4}:
\begin{exam}
Let $X=\{x\}$, let $E=\{e_1,e_2\}$, and let $\bs=\{\s_{e_1},\s_{e_2}\}$, where $\s_{e_1}=\s_{e_2}=\{\emptyset, X\}$. One can check that each $\s_{e_i}$ is trivially a regular space. On the other hand, the soft topology $\T(\bs)=\{\widetilde{\Phi}, (F_1,E), (F_2,E), \widetilde{X}\}$ is not soft regular, where $(F_1,E)=\{(e_1, \emptyset), (e_2, X)\}$ and $(F_2,E)=\{(e_1, X), (e_2, \emptyset)\}$. Indeed, $x\notin (F_i,E)^c$ for each $i$, but no soft open sets in $\T(\bs)$ can separate them.
\end{exam}

A less trivial example is following:

\begin{exam}
Let $X=\mathbb{R}$ be the set of reals, let $E=\{e_1,e_2\}$, and let $\bs=\{\s_{e_1},\s_{e_2}\}$, where $\s_{e_1}$ is the natural topology and $\s_{e_2}$ is the Sorgenfrey line on $\mathbb{R}$. It is known that both $\s_{e_1}$ and $\s_{e_2}$ are regular spaces, while $\T(\bs)$ is not soft regular. Take $x\in\mathbb{R}$ and two closed sets $[a,b], [c,d)$ for which $x\in[a,b]$ and $x\notin[c,d)$. If $(F,E)=\{(e_1, [a,b]), (e_2, [c,d))\}$, then $(F,E)$ is a soft closed set in $\T(\bs)$ such that $x\notin (F,E)$. One can easily check that $x$ and $(F,E)$ cannot be separated by soft open sets.
\end{exam}
We shall admit that the idea of this example is due to Terepeta \cite[Example 8]{terepeta2019separating}.

\begin{thm}\label{thm5}
Let $\bs=\{\s_e:e\in E\}$ be a family of crisp topologies on $X$. Then $\s_e$ is a normal space for each $e\in E$ if and only if $\T(\bs)$ is a soft normal space.
\end{thm}
\begin{proof}
Let $\s_e$ be normal for each $e\in E$. Suppose $(A,E), (B,E)$ are disjoint soft closed sets in $\T(\bs)$. Then $(A,E)=\{(e,A(e)):A^c(e)\in\s_e, e\in E\}$ and $(B,E)=\{(e,B(e)):B^c(e)\in\s_e, e\in E\}$. Therefore $\widetilde{\Phi}=(A,E)\widetilde{\bigcap}(B,E)=\{(e,A(e)\cap B(e)):A^c(e), B^c(e)\in\s_e, e\in E\}$. We obtain that $A(e)\cap B(e)=\emptyset$. Since $\s_e$ is a normal space for each $e\in E$, there exist open sets $G(e),H(e)$ such that $A(e)\subseteq G(e)$, $B(e)\subseteq H(e)$ and $G(e)\cap H(e)=\emptyset$. Set $(G,E)=\{(e,G(e)): e\in E\}$ and $(H,E)=\{(e,H(e)): e\in E\}$. Then $(G,E), (H,E)\in\T(\bs)$ such that $(A,E)\widetilde{\subseteq}(G,E)$ and $(B,E)\widetilde{\subseteq}(H,E)$. Furthermore, $(G,E)\widetilde{\bigcap}(H,E)=\{(e,G(e)\cap H(e)): e\in E\}=\widetilde{\Phi}$. This shows that $\T(\bs)$ is soft normal.

Conversely, for each $e\in E$, we let $C(e), D(e)$ be disjoint closed sets in $\s_e$. By Lemma \ref{closed=softclosed}, $(C,E)=\{(e,C(e)):C^c(e)\in\s_e, e\in E\}$ and $(D,E)=\{(e,D(e)):D^c(e)\in\s_e, e\in E\}$ are soft closed sets in $\T(\bs)$ and  $(C,E)\widetilde{\bigcap}(D,E)=\{(e,C(e)\cap D(e)): e\in E\}=\widetilde{\Phi}$. Since $\T(\bs)$ is soft normal, then there exist disjoint soft open sets $(U,E), (V,E)$ such that $(C,E)\widetilde{\subseteq}(U,E)$ and $(D,E)\widetilde{\subseteq}(V,E)$. This implies that  $C(e)\subseteq U(e)$, $D(e)\subseteq V(e)$ and $U(e)\cap V(e)=\emptyset$ for each $e\in E$. Thus, $\s_e$ is normal for each $e\in E$.
\end{proof}

We shall remark that if $(X,\s,E)$ is a soft normal space, then $\s_e$ need not be a normal space for each $e\in E$.
\begin{exam}
Let $X=\{x_1,x_2,x_3\}$ and $E=\{e_1,e_2\}$. Suppose $$\s=\big\{\widetilde{\Phi},(H_1,E),(H_2,E), \cdots,(H_9,E),\widetilde{X}\big\},$$
where
\vspace{-3mm}
\begin{flalign*}
	(H_1,E)&=\{(e_1,\emptyset), (e_2,\{x_3\})\},&&\\
	(H_2,E)&=\{(e_1,\{x_3\}), (e_2,\emptyset)\},&&\\
	(H_3,E)&=\{(e_1,\{x_2,x_3\}), (e_2,\emptyset)\},&&\\
	(H_4,E)&=\{(e_1,X), (e_2, \emptyset)\},&&\\
	(H_5,E)&=\{(e_1,\{x_1,x_3\}), (e_2, \emptyset)\},&&\\
	(H_6,E)&=\{(e_1,\{x_3\}), (e_2, \{x_3\})\},&&\\
	(H_7,E)&=\{(e_1,\{x_2,x_3\}), (e_2, \{x_3\})\},&&\\
	(H_8,E)&=\{(e_1,X), (e_2, \{x_3\})\}, and&&\\
	(H_9,E)&=\{(e_1,\{x_1,x_3\}), (e_2, \{x_3\})\}.&&
\end{flalign*} 
Then $\s$ is a soft normal space as each pair of non-null soft closed sets intersects each other. On the other hand, $\s_{e_1}$ is not normal. 
\end{exam}

The following example demonstrates that if one $\s_e$ is not normal, then $\T(\bs)$ needs not be soft normal:
\begin{exam}\label{milan2}
Let $X=\{x_1,x_2,x_3\}$ and $E=\{e_1,e_2\}$. Take $\s_{e_1}=\{\emptyset, X\}$ and $\s_{e_2}=\{\emptyset, \{x_1\}, \{x_1, x_2\}, \{x_1, x_3\}, X\}$. One can easily verify that $\s_{e_1}$ is normal but not $\s_{e_2}$. The closed sets $\{x_2\}, \{x_3\}$ in $\s_{e_2}$ cannot be separated. The soft topology $$\T(\bs)=\big\{\widetilde{\Phi}, (F_1,E), (F_2,E), (F_3,E), (F_4,E), (F_5,E), (F_6,E), (F_7,E), (F_8,E), \widetilde{X}\big\},$$ 
where 
\vspace{-3mm}
\begin{flalign*}
(F_1,E)&=\{(e_1,\emptyset), (e_2,\{x_1\})\},&&\\
(F_2,E)&=\{(e_1,\emptyset), (e_2,\{x_1, x_2\})\},&&\\
(F_3,E)&=\{(e_1,\emptyset), (e_2,\{x_1, x_3\})\},&&\\
(F_4,E)&=\{(e_1,\emptyset), (e_2, X)\},&&\\
(F_5,E)&=\{(e_1,X), (e_2, \emptyset)\},&&\\
(F_6,E)&=\{(e_1,X), (e_2, \{x_1\})\},&&\\
(F_7,E)&=\{(e_1,X), (e_2, \{x_1, x_2\})\},\ and&&\\
(F_8,E)&=\{(e_1,X), (e_2, \{x_1, x_3\})\},&&
\end{flalign*} 
is not a soft normal space. Indeed, since $\{(e_1, \emptyset), (e_2,\{x_3\})\}$ and  $\{(e_1, X), (e_2,\{x_2\})\}$ are disjoint soft closed sets in $\T(\bs)$ but no soft open sets can separate them. 
\end{exam}

\begin{rem}
Notice that we can provide a more general proof sketch that proves the above claim, which is another proof to the part one of Theorem \ref{thm5}. If $\bs=\{\s_e:e\in E\}$ is not a normal topology for some $\bar{e}\in E$, then there are closed sets $A ,B$ in $\s_{\bar{e}}$ which cannot be separated by any open sets. The soft sets $(F^{\bar{e}}_A, E), (F_{\bar{e}}^B, E)$ are closed and disjoint in $\T(\bs)$. If there exist soft open sets $(G,E), (H,E)$ in $\T(\bs)$ such that $(F^{\bar{e}}_A, E)\widetilde{\subseteq}(G,E)$, $(F_{\bar{e}}^B, E)\widetilde{\subseteq}(H,E)$ such that $(G,E)\widetilde{\bigcap}(H,E)=\widetilde{\Phi}$. This concludes that $A\subseteq G(\bar{e})$, $B\subseteq H(\bar{e})$ and $G(\bar{e})\cap H(\bar{e})=\emptyset$, a contradiction. 

\end{rem}

\section*{Conclusion}
This study develops a methodical understanding of the connections between a system of crisp topologies and the soft topology produced by it. The procedure is carried out with the help of two formulas. If we start from an original soft topology, then it produces a system of crisp topologies. With our formulas, we can generate two different new soft topologies. We have discussed the relationships between these soft topologies. Moreover, we show that the resulting soft topology via Formula 1 is always finer than the original one, while the soft topology generated by Formula 2 is incomparable. We see that two different original soft topologies may generate a single soft topology by either of the formulas. Furthermore, we study the
 preservation of separation axioms between the system of crisp topologies and the soft topology generated by it. More precisely, we show that Hausdorffness and normality behave better in transforming to soft topologies and conversely. On the other hand, other separation axioms act differently. If one of the crisp topologies is respectively $T_0$, $T_1$, then it guarantees that the resulting soft soft topology is soft $T_0$, $T_1$. The converse is also true for $T_0$. All of the crisp topologies are regular when the soft topology generated by them is soft regular.
\section*{Acknowledgment}
We would like to thank Milan Matejdes for Examples \ref{ex1}\&\ref{milan2}.

\section*{Ethical approval}
This article does not contain any studies with human participants or animals performed by the author.

\section*{Funding details}
This article received no external funding.

\section*{Conflict of interest}
The author declare no conflict of interest.

\section*{Availability of data and materials}
Data sharing is not applicable to this article as no data sets were generated or analyzed during the current study.

\section*{Authorship contribution}
The authors dealt with the conceptualization, formal analysis, supervision, methodology, investigation, and writing original draft preparation. They also contributed equally in the formal analysis; writing, review and editing. They read and approved the final version of the article.

\end{document}